\documentclass[reqno,12pt]{amsart}
\usepackage{amssymb,latexsym}
\usepackage{array,multirow,makecell}
\usepackage{mathtools}
\usepackage{IEEEtrantools}
\usepackage{amsmath}
\setcellgapes{1pt}
\makegapedcells
\newtheorem{theorem}{Theorem}[section]
\newtheorem{corollary}[theorem]{Corollary}

\newtheorem{proposition}[theorem]{Proposition}
\newtheorem{korselt's criterion}[theorem]{Korselt's criterion}
\newtheorem{definition}[theorem]{Definition}

\newtheorem{remark}[theorem]{Remark}

\newtheorem{example}[theorem]{Example}

\newtheorem{conjecture}[theorem]{Conjecture}

\newtheorem{proof of  Main Theorem}[theorem]{Proof of  Main Theorem}
\numberwithin{equation}{section}
\begin{document}
\title[Connections on the Rational Korselt Set of $pq$ ]
{Connections on the Rational Korselt Set of $pq$}

\author{Nejib Ghanmi}
\address[Ghanmi]{Preparatory Institute of Engineering Studies, Tunis university, Tunisia.}
\address[]{\hspace{1.5cm}Current Address: University College of Jammum, Department of Mathematics, Mekkah, Saudi Arabia.}

\email{naghanmi@uqu.edu.sa\; and\; neghanmi@yahoo.fr}

\thanks{}



\begin{abstract} For  a positive integer $N$ and $\mathbb{A}$ a subset of $\mathbb{Q}$, let  $\mathbb{A}$-$\mathcal{KS}(N)$ denote the set of $\alpha=\dfrac{\alpha_{1}}{\alpha_{2}}\in \mathbb{A}\setminus \{0,N\}$   verifying $\alpha_{2}r-\alpha_{1}$ divides $\alpha_{2}N-\alpha_{1}$ for every prime divisor $r$ of $N$. The set  $\mathbb{A}$-$\mathcal{KS}(N)$ is   called the set of $N$-Korselt bases  in $\mathbb{A}$.

 Let  $p, q$ be two  distinct prime numbers. In this paper, we prove that each $pq$-Korselt base  in $\mathbb{Z}\setminus\{ q+p-1\}$  generates other(s) in $\mathbb{Q}$-$\mathcal{KS}(pq)$. More precisely, we will prove that  if   $(\mathbb{Q}\setminus\mathbb{Z})$-$\mathcal{KS}(pq)=\emptyset$ then $\mathbb{Z}$-$\mathcal{KS}(pq)=\{ q+p-1\}$.

\end{abstract}

%%% ----------------------------------------------------------------------
\maketitle
%%% ----------------------------------------------------------------------

\section{Introduction}

A Carmichael number $N$ is a positive composite integer that satisfies $a^{N}\equiv 1 \, mod \, N $; for any $a$ with $gcd(a; N) = 1$: It follows that a Carmichael number $N$ meets Korselt's
criterion:
\begin{korselt's criterion}[\cite{Kor}]
A composite integer  $N>1$ is a Carmichael number if and only if $p-1$ divides $N-1$ for all prime factors  $p$ of $N$ .

\end{korselt's criterion}

In ~\cite{ BouEchPin, echi}, Bouall\`egue-Echi-Pinch introduced the notion of $\alpha$-Korselt numbers where $\alpha\in\mathbb{Z}\setminus \{0\} $ as a generalization of Carmichael numbers as follows.
\begin{definition}
An $\alpha$-Korselt number is a number $N$ such that $p-\alpha$ divides $N-\alpha$ for all $p$ prime divisor of $N$.
\end{definition}
 The $\alpha$-Korselt numbers for $\alpha \in \mathbb{Z}$ are well investigated last years specially in ~\cite{ KorGhan, BouEchPin, Ghanmi, echi, WilGhan}. In ~\cite{Ghanmi2},  Ghanmi proposed  another generalization for $\alpha=\dfrac{\alpha_{1}}{\alpha_{2}}\in\mathbb{Q}\setminus\{0\}$  by setting the following definitions.

\begin{definition} Let $N\in \mathbb{N}\setminus\{0,1\}$,  $\alpha=\dfrac{\alpha_{1}}{\alpha_{2}}\in \mathbb{Q}\setminus \{0\}$ with $gcd(\alpha_{1}, \alpha_{2})=1$ and $\mathbb{A}$ a subset of $\mathbb{Q}$.
Then
\begin{enumerate}
  \item $N$ is said to be an \emph{$\alpha$-Korselt number\index{Korselt number}} (\emph{$K_{\alpha}$-number}, for
short), if $N\neq \alpha$ and $\alpha_{2}p-\alpha_{1}$ divides $\alpha_{2}N-\alpha_{1}$ for
every prime divisor $p$ of $N$.

\item By the \emph{$\mathbb{A}$-Korselt set}\index{Korselt set} of the number $N$ (or the Korselt set of $N$ over \emph{$\mathbb{A}$}), we mean the set $\mathbb{A}$-$\mathcal{KS}(N)$ of
all $\beta\in \mathbb{A}\setminus\{0,N\}$ such that $N$ is a $K_{\beta}$-number.
  \item The cardinality of $\mathbb{A}$-$\mathcal{KS}(N)$ will be called the \emph{$\mathbb{A}$-Korselt
weight}\index{Korselt weight} of $N$; we denote it by $\mathbb{A}$-$\mathcal{KW}(N)$.

\end{enumerate}

\end{definition}
 Carmichael numbers are exactly the 1-Korselt  squarefree composite numbers.
   Further, in ~\cite{Ghanmi3,Ghanmi4} Ghanmi set  the notion of Korselt bases as follows.

\begin{definition} Let $N\in \mathbb{N}\setminus\{0,1\}$, $\alpha\in \mathbb{Q}$ and $\mathbb{B}$ be a subset of $\mathbb{N}$. Then

\begin{enumerate}
\item  $\alpha$ is called \emph{$N$-Korselt base\index{Korselt base}}(\emph{$K_{N}$-base}, for
short), if $N$ is a \emph{$K_{\alpha}$-number}.
\item By the \emph{$\mathbb{B}$-Korselt set}\index{Korselt base set} of the base $\alpha$ (or the Korselt set of the base $\alpha$ over \emph{$\mathbb{B}$}), we mean the set $\mathbb{B}$-$\mathcal{KS}(B(\alpha))$ of
all $M\in \mathbb{B}$ such that $\alpha$ is a
$K_{M}$-base.
\item The cardinality of  $\mathbb{B}$-$\mathcal{KS}(B(\alpha))$ will be called the \emph{$\mathbb{B}$-Korselt
weight}\index{Korselt  weight} of the base $\alpha$; we denote it by $\mathbb{B}$-$\mathcal{KW}(B(\alpha))$.
\end{enumerate}
\end{definition}

 The set $\mathbb{Q} \text{-}\mathcal{KS}(N)$ is simply called the rational Korselt set of $N$.

In this paper we are concerned only with a squarefree composite number $N$.

After extending  the notion of Korselt numbers to $ \mathbb{Q}$ and in order to study the Korselt numbers and their korselt sets  over $\mathbb{Q}$, a natural query can be asked about the existence of some connections between the Korselt bases of a number $N$ over the sets $\mathbb{Z}$ and $ \mathbb{Q\setminus}\mathbb{Z}$. The answer is affirmative for a squarefree composite number $N$ with two prime factors. Indeed, if we look deeply at a list of Korselt numbers and their Korselt sets [ see Table $1$ and Table $2$], we remark the nonexistence  of any squarefree composite number $N$ with two prime factors such  that $\mathbb{Z}$-$\mathcal{KW}(N)\geq 2$ and $(\mathbb{Q}\setminus\mathbb{Z})$-$\mathcal{KS}(N)=\emptyset$. This inspired us to claim that such relation between  $\mathbb{Z}$-$\mathcal{KS}(N)$ and $(\mathbb{Q}\setminus\mathbb{Z})$-$\mathcal{KS}(N)$ exists.   The case when $N$ is squarefree with more than two prime factors is still untreated. In order to display this  (these) connection(s) we  organize our work in this paper  as  follows :

- In Section $2$, we give some numerical data showing  connections between the Korselt bases of $N$ over $\mathbb{Z}$ and $\mathbb{Q}\setminus\mathbb{Z}$.

- In Section $3$, we prove  that for each squarefree composite  number $N$ with two prime factors,  some $N$-Korselt bases in $\mathbb{Z}$  generate others in the same set $\mathbb{Z}$-$\mathcal{KS}(N)$.

 - In Section $4$,  we show that for each squarefree composite  number $N=pq$ with two prime factors,  each $N$-Korselt bases in $\mathbb{Z}\setminus\{q+p-1\}$  generates a Korselt base   in $ \mathbb{Q}\setminus\mathbb{Z}$.

\section{Preliminaries}
 The following data in  Table $1$ and Table $2$ illustrate some cases for Korselt numbers and sets:

 -- Table $1$ gives the set  $\mathbb{Z}$-$\mathcal{KS}(N)$ for  numbers $N=pq$ with $p, q$ primes and $p<q\leq53$ such that  $(\mathbb{Q}\setminus\mathbb{Z})$-$\mathcal{KS}(N)=\emptyset$.

-- Table $2$ gives, for each integer $1\leq i\leq 7$,  the smallest squarefree composite number $N_i=pq$ with $p,q$ primes, $p<q<10^{3}$   such that  $\mathbb{Z}$-$\mathcal{KW}(N_i)=i$ and  $(\mathbb{Q}\setminus\mathbb{Z})$-$\mathcal{KW}(N_i)$ be the smallest.

\begin{center}
\medskip
{\small
\setcellgapes{3pt}
\begin{tabular}{|c|c||c|c|} \hline
$N$ & $\mathbb{Z}$-$\mathcal{KS}(N)$&$N$ & $\mathbb{Z}$-$\mathcal{KS}(N)$ \\ \hline
 $2*11$ &$\{12 \}$ & $2*41$  &$\{42\}$     \\
 $2*13$ & $\{14\}$ &$3*41$  &$\{43\}$ \\
 $2*17$  & $\{18\}$ &$5*41$& $\{45\}$ \\
 $2*19$  & $\{20\}$ & $2*43$  &  $\{44\}$\\
 $3*19$  & $\{21\}$ & $3*43$  &$\{45\}$\\
 $2*23$  & $\{24\}$ &$5*43$  &$\{47\}$\\
 $3*23$  & $\{25\}$ & $2*47$  &$\{48\}$\\
 $2*29$ &  $\{30\}$&$3*47$  &$\{49\}$\\
 $3*29$&  $\{31\}$&$5*47$ &$\{51\}$\\
 $2*31$  &  $\{32\}$&$13*47$&$\{59\}$\\
 $3*31$& $\{33\}$&$2*53$&$\{54\}$\\
  $2*37$ & $\{38\}$&$3*53$&$\{55\}$\\
  $3*37$  &$\{39\}$& $5*53$&$\{57\}$\\ \hline
  \end{tabular}

\medskip
\textbf{Table $1$.} $\mathbb{Z}$-$\mathcal{KS}(N)$ where $N=pq; p,q$ primes , $p<q\leq53$ and  $(\mathbb{Q}\setminus\mathbb{Z})$-$\mathcal{KS}(N)=\emptyset.$
}\end{center}

\begin{center}

\bigskip
{\small
\setcellgapes{4pt}
\begin{tabular}{|l|l|l||l|l|l|} \hline

$i$ &$N_i$ &$\mathbb{Z}$-$\mathcal{KS}(N_i)$  &$(\mathbb{Q}\setminus\mathbb{Z})$-$\mathcal{KW}(N_i)$ \\ \hline

1 & $2*11$& $\{12\}$& $0$ \\ \hline

2& $2*7$&$\{6,8\}$&$1$\\ \hline

3  &$5*19$&$\{15,20,23\}$& $2$\\ \hline
4&$31*59$&$\{29,60,62,89\}$&$5$\\ \hline
5& $67*97$&$\{64, 75, 91, 99, 163\}$&$12$ \\ \hline

6& $757*881$&$\{755, 773, 797, 845, 867, 1637\}$&$17$ \\ \hline
7& $37*61$&$\{25, 43, 49, 52, 57, 67, 97\}$&$22$ \\ \hline

\end{tabular}

\bigskip
\textbf{Table $2$.} The smallest  $N_i=pq$ with $p,q$ primes, $p<q<10^{3}$  such that  $\mathbb{Z}$-$\mathcal{KW}(N_i)=i$ and  $(\mathbb{Q}\setminus\mathbb{Z})$-$\mathcal{KW}(N_i)$ be the smallest.
}\end{center}

\bigskip

From Table $1$ and Table $2$ we remark that there is no squarefree composite  number $N$  with two prime factors  such  that $\mathbb{Z}$-$\mathcal{KW}(N)\geq 2$ and $(\mathbb{Q}\setminus\mathbb{Z})$-$\mathcal{KS}(N)=\emptyset$, this equivalent to establish the following result:
\begin{theorem}[\textbf{Main Theorem}]\label{fond}
  Let $N=pq$. If $(\mathbb{Q}\setminus\mathbb{Z})$-$\mathcal{KS}(N)=\emptyset$ then $\mathbb{Z}$-$\mathcal{KS}(N)=\{ q+p-1\}$.
\end{theorem}

  Moreover, it appear that for these numbers, the sets  $\mathbb{Z}$-$\mathcal{KS}(N)$ and  $(\mathbb{Q}\setminus\mathbb{Z})$-$\mathcal{KS}(N)$ are somehow related. To display this relation, we will show that each $N$-Korselt base in  $\mathbb{Z}\setminus\{p+q-1\}$ induces at least an other $N$-Korselt base in $(\mathbb{Q}\setminus\mathbb{Z})$-$\mathcal{KS}(N)$. Hence, the Main Theorem will be deduced immediately.

For all the rest of this paper let $p<q$ be two primes, $N=pq$ and $i,s$ be the integers given by the division algorithm of
$q$ by $p$: $q=ip+s$ with $s \in \{1, \ldots, p-1\}$.

Our work is based on the following result given  by Echi-Ghanmi~\cite{Ghanmi}.

  \begin{theorem}[]\label{structure1}

  Let $N=pq$ such that $p<q$, the following properties hold.
\begin{itemize}
    \item[$(1)$] If $q>2p^2$, then $\mathbb{Z}\text{-}\mathcal{KS}(N)=\{p+q-1\}$.
    \item[$(2)$] If $p^2-p<q<2p^2$ and $p\geq 5$, then $$\mathbb{Z}\text{-}\mathcal{KS}(N)\subseteq\{ip, p+q-1\}.$$
    \item[$(3)$] If $4p<q<p^2-p$, then
    $$\mathbb{Z}\text{-}\mathcal{KS}(N)\subseteq\{ip,(i+1)p,p+q-1\}.$$
    \item[$(4)$] Suppose that $3p<q<4p$. Then the following
    conditions are satisfied.

\begin{itemize}
       \item[$(a)$] If $q=4p-3$, then the
following properties hold.
\begin{itemize}
    \item[$(i)$] If $p\equiv1 \pmod 3$, then $$\mathbb{Z}\text{-}\mathcal{KS}(N)=\{4p, q-p+1, p+q-1\}.$$
    \item[$(ii)$] If $p\not\equiv1 \pmod 3$, then $$\mathbb{Z}\text{-}\mathcal{KS}(N)=\{q-p+1, p+q-1\}.$$

\end{itemize}
       \item[$(b)$] If $q\neq 4p-3$, then
    $$\mathbb{Z}\text{-}\mathcal{KS}(N)\subseteq\{3p,4p,p+q-1\}.$$
    \end{itemize}
    \item[$(5)$] If $2p<q<3p$, then $$\mathbb{Z}\text{-}\mathcal{KS}(N)\subseteq\{2p,3p,3q-5p+3,\dfrac{2p+q-1}{2},q-p+1,p+q-1\}.$$
\item[$(6)$] If $p<q<2p$, then $$\mathbb{Z}\text{-}\mathcal{KS}(N)\subseteq \{q+p-1\}\cup[2,2p]\setminus\{p\}.$$
\end{itemize}\end{theorem}

\bigskip

\section{Connections in $\mathbb{Z}$-$\mathcal{KS}(N)$  }

By the following result, we prove that certain $N$-Korselt bases   in $\mathbb{Z}$  induce others in the same set $\mathbb{Z}$-$\mathcal{KS}(N)$.
\begin{proposition}\label{gener1}

 Suppose that $2p<q<3p$. Then  the following properties hold.

\begin{enumerate}
                     \item   $\dfrac{2p+q-1}{2}\in\mathbb{Z}$-$\mathcal{KS}(N)$ if and only if $q-p+1\in\mathbb{Z}$-$\mathcal{KS}(N)$.
                     \item If  $3q-5p+3\in\mathbb{Z}$-$\mathcal{KS}(N)$ then $q-p+1\in\mathbb{Z}$-$\mathcal{KS}(N)$.
                   \end{enumerate}

\end{proposition}

\begin{proof}
First, since $q=2p+s$, the integer $s$ must be odd and so  $s<p-1$.

\begin{enumerate}

  \item We have $\alpha=\dfrac{2p+q-1}{2}\in\mathbb{Z}$-$\mathcal{KS}(N) $ if and only if

  $$\left\{\begin{array}{rrr}
p-\alpha=\dfrac{-q+1}{2} & \mid & p(q-1)\\
q-\alpha=\dfrac{s+1}{2}  & \mid & q(p-1)\\
\end{array}
\right. $$

\medskip

     which is equivalent to $s+1$  divides  $2q(p-1).$

   As in addition $s< p-1<q-1$ hence $\gcd(q,s+1)=1$,  and

   $2(p-1)=q-1-(s+1)$, it follows that

\begin{equation}\label{eq6} \dfrac{2p+q-1}{2}\in\mathbb{Z}\text{-}\mathcal{KS}(N) \ \  \text{if and only if } \ \ s+1\, \mid\, q-1.\end{equation}
   On the other hand $\beta=q-p+1\in\mathbb{Z}$-$\mathcal{KS}(N) $ is equivalent to

    $$\left\{\begin{array}{rrr}
p-\beta=-s-1 & \mid & p(q-1)\\
q-\beta=p-1  & \mid & q(p-1)\\
\end{array}
\right. $$

    which is equivalent to $s+1$ divides $ p(q-1).$

    As in addition $s< p-1$ hence $gcd(p,s+1)=1$, it follows that
\begin{equation}\label{eq7}q-p+1\in\mathbb{Z}\text{-}\mathcal{KS}(N) \ \  \text{if and only if } \ \ s+1\, \mid\, q-1.\end{equation}

     So, by $\eqref{eq6}$ and $\eqref{eq7}$, we conclude that
     $$ \dfrac{2p+q-1}{2}\in\mathbb{Z}\text{-}\mathcal{KS}(N) \ \  \text{if and only if } \ \ q-p+1\in\mathbb{Z}\text{-}\mathcal{KS}(N).$$
   \item Suppose that  $\gamma=3q-5p+3 \in\mathbb{Z}$-$\mathcal{KS}(N)$. Then,
\begin{equation}\label{eq8}p-\gamma=6p-3q-3=-3(s+1)\, \mid \, p(q-1).\end{equation}
     We consider two cases:

    \begin{itemize}
     \item If $p\neq3$, then since $s< p-1$, we have $gcd(p,3(s+1))=1$. Hence  by $\eqref{eq8}$, $ 3(s+1)$ divides $q-1$. Thus by $\eqref{eq7}$, $q-p+1\in\mathbb{Z}$-$\mathcal{KS}(N)$.

      \item Now, assume that $p=3$. First, as $1\leq s\leq p-2=1$, we get $s=1$,  $q=2p+s=7$ and $q-p+1=5$. So, easily we can check that $N=3*7=21$ is a $5$-Korselt number.

\end{itemize}

\end{enumerate}
\end{proof}
\begin{corollary}\label{congene}
 If $q>2p$ and $q-p+1 \notin \mathbb{Z}$-$\mathcal{KS}(N)$, then
 $$\mathbb{Z}\text{-}\mathcal{KS}(N)\subseteq\{ip, (i+1)p,p+q-1\}$$
\end{corollary}

\begin{proof}
By Theorem ~\ref{structure1}, the answer is straight forward when $q>3p$.

Now, suppose that  $2p<q<3p$ (i.e. $i=2$).

Let $\beta\in\mathbb{Z}\text{-}\mathcal{KS}(N)$. Then, again by Theorem ~\ref{structure1}, we have

$$\beta\in\{2p,3p,3q-5p+3,\dfrac{2p+q-1}{2},q-p+1,p+q-1\}.$$

But, since $q-p+1 \notin \mathbb{Z}$-$\mathcal{KS}(N)$, we have by Proposition~\ref{gener1}

$\beta\neq 3q-5p+3,\dfrac{2p+q-1}{2}$. Thus, $\beta\in\{2p,3p,p+q-1\}$ which  what is required.

\end{proof}
\section{Connections between $\mathbb{Z}$-$\mathcal{KS}(N)$ and $(\mathbb{Q}\setminus\mathbb{Z})$-$\mathcal{KS}(N)$ }
The following result concern the case when $q<2p$.
\begin{proposition}\label{gene1}

Suppose that $q=p+s$ and $\beta \in \mathbb{Z}$ with $\beta\neq p+q-1$ and
 $gcd(p,\beta)=gcd(pq,p+q-\beta)=1$.  Then, $\beta \in \mathbb{Z}$-$\mathcal{KS}(N)$ if and only if $\dfrac{qp}{p+q-\beta}\in(\mathbb{Q}\setminus\mathbb{Z})$-$\mathcal{KS}(N)$.

\end{proposition}

\begin{proof}

  As $gcd(p,\beta)=1$ and  by~\cite[Proposition $4$]{Ghanmi} $gcd(q,\beta)=1$, we have
   $$\hspace{-3.9cm}\beta\in\mathbb{Z}\text{-}\mathcal{KS}(N)\Leftrightarrow\left\{\begin{array}{rrr}
p-\beta & \mid & q-1\\
q-\beta   & \mid & p-1\\
\end{array}
\right. $$

$$\hspace{3.5cm}\Leftrightarrow\left\{\begin{array}{rrr}
(p+q-\beta)p-pq= (p-\beta)p & \mid & p(q-1)\\
(p+q-\beta)q-pq=(q-\beta)q & \mid & q(p-1)\\
\end{array}
\right. $$
 \hspace{3.6cm} $\Leftrightarrow\dfrac{qp}{p+q-\beta}\in\mathbb{Q}$-$\mathcal{KS}(N)$.

As $\beta \notin \{p,q\}$ then $p+q-\beta \notin \{p,q\}$ .
Further, since $p<q<2p$,   we have by Theorem ~\ref{structure1}, $2\leq\beta<2p$, hence $p+q-\beta \geq 2p-\beta+1\geq2$,  that is $p+q-\beta \neq1$. So, $\dfrac{qp}{p+q-\beta}\notin\mathbb{Z}$ and we conclude that $$\dfrac{qp}{p+q-\beta}\in(\mathbb{Q}\setminus\mathbb{Z})\text{-}\mathcal{KS}(N).$$
\end{proof}

The next two results, concern the case when $p$ divides $\beta$.

\begin{proposition}\label{gene2}
If  $ip\in\mathbb{Z}$-$\mathcal{KS}(N)$,  then there exists $k_{1}\in\mathbb{N}\setminus\{0,1\}$ such that $\dfrac{(k_{1}+1)q}{i k_{1}+1}\in (\mathbb{Q}\setminus\mathbb{Z})$-$\mathcal{KS}(N)$ .
\end{proposition}

\begin{proof}

Let $\beta=ip\in\mathbb{Z}$-$\mathcal{KS}(N)$. Then
$$\left\{\begin{array}{rrr}
p-ip  & \mid & pq-ip=p(q-1)+p-ip\\
q-ip  & \mid & pq-ip=q(p-1)+q-ip\\
\end{array}
\right. $$
As $gcd(s,q)=1$, this is equivalent to
$$\left\{\begin{array}{rrr}
i-1  & \mid & q-1\\
s  & \mid & p-1\\
\end{array}
\right. $$
hence, there exist $k_{1}$ and $k_{2}$ in $\mathbb{Z}$ such that

$$\left\{\begin{array}{rrl}
 q-1& =& k_{2}(i-1)\\
p-1 & = & k_{1}s\\
\end{array}
\right. $$

As $q=ip+s$ then $k_{1}q=i k_{1}p+k_{1}s=i k_{1}p+p-1$, therefore \begin{equation}\label{eq9}(k_{1}+1)q-(i k_{1}+1)p=q-1.\end{equation}
Let $k=gcd(k_{1}+1,i k_{1}+1)$, $\alpha_{1}^{'}=\dfrac{k_{1}+1}{k}$ and $\alpha_{2}=\dfrac{i k_{1}+1}{k}$.

So, we get by $\eqref{eq9}$ \begin{equation}\label{eq10}\alpha_{1}^{'}q-\alpha_{2}p=\dfrac{q-1}{k}.\end{equation}

Now, let us prove that $\alpha_{2}-\alpha_{1}^{'} $ divides $p-1$.

First, note that \begin{equation}\label{eq100}\alpha_{2}-\alpha_{1}^{'}=\dfrac{k_{1}}{k}(i-1).\end{equation}
Since $q-1=(i-1)p+(k_{1}+1)s$ and $i-1 \mid q-1$, we deduce that $i-1\mid (k_{1}+1)s$.
Further, as $gcd(k_{1}+1,i-1)=gcd(k_{1}+1,i k_{1}+1)=k$, it follows that $m=\dfrac{i-1}{k}\mid \dfrac{k_{1}+1}{k}s$.
But $gcd(\dfrac{k_{1}+1}{k},\dfrac{i-1}{k})=1$, hence $m\mid s$. So, we conclude  by $\eqref{eq100}$, that

\begin{equation}\label{eq11}\alpha_{2}-\alpha_{1}^{'}=k_{1}m \mid k_{1}s=p-1.\end{equation}
Now, by $\eqref{eq10}$ and $\eqref{eq11}$, we obtain

$$\left\{\begin{array}{rrr}
\alpha_{2}p-\alpha_{1}^{'}q  & \mid & q-1\\
\alpha_{2}-\alpha_{1}^{'}  & \mid & p-1\\
\end{array}
\right. $$

Thus, $$\alpha=\dfrac{\alpha_{1}^{'}q }{\alpha_{2}}=\dfrac{(k_{1}+1)q}{i k_{1}+1}\in \mathbb{Q}\text{-}\mathcal{KS}(N)$$

As in addition   $gcd(\alpha_{1}^{'},\alpha_{2})=1$, $gcd(q,\alpha_{2})=1$ by $\eqref{eq10}$ and $\alpha_{2}\neq1$, we conclude that  $$\dfrac{(k_{1}+1)q}{i k_{1}+1}\in (\mathbb{Q}\setminus\mathbb{Z})\text{-}\mathcal{KS}(N).$$

\end{proof}

\begin{proposition}\label{gene3}

If $(i+1)p\in\mathbb{Z}$-$\mathcal{KS}(N)$ and $s>1$, then there exists
$k_{1}\in\mathbb{N}\setminus\{0,1\}$ such that $\dfrac{(k_{1}-1)q}{(i+1)k_{1}-1}\in (\mathbb{Q}\setminus\mathbb{Z})$-$\mathcal{KS}(N)$ .

\end{proposition}
\begin{proof}

 If $\beta=(i+1)p\in\mathbb{Z}$-$\mathcal{KS}(N)$,  then
$$\left\{\begin{array}{rrr}
p-(i+1)p  & \mid & pq-(i+1)p =p(q-1)+p-(i+1)p\\
q-(i+1)p  & \mid & pq-(i+1)p=q(p-1)+q-(i+1)p\\
\end{array}
\right. $$
This is equivalent to
$$\left\{\begin{array}{rrr}
i  & \mid & q-1\\
p-s  & \mid & p-1\\
\end{array}
\right. $$
hence, there exist $k_{1}$ and $k_{2}$ in $\mathbb{N}\setminus\{0\}$ such that

$$\left\{\begin{array}{rrl}
 q-1& =& k_{2}i\\
p-1 & = & k_{1}(p-s)\\
\end{array}
\right. $$

First, as $s>1$ it follows that $k_{1}>1$. Since $q=(i+1)p+s-p$, hence $k_{1}q=(i+1)k_{1}p-p+1$. Therefore, we can write \begin{equation}\label{eq12}((i+1)k_{1}-1)p-(k_{1}-1)q=q-1.\end{equation}
Let $k=gcd(k_{1}-1,(i+1)k_{1}-1)$, $\alpha_{1}^{'}=\dfrac{k_{1}-1}{k}$ and $\alpha_{2}=\dfrac{(i+1)k_{1}-1}{k}$.

So, we get by $\eqref{eq12}$
\begin{equation}\label{eq13}\alpha_{2}p-\alpha_{1}^{'}q=\dfrac{q-1}{k}.\end{equation}

Now, let us prove that $\alpha_{2}-\alpha_{1}^{'} \mid p-1$.

First, we have \begin{equation}\label{eq130}\alpha_{2}-\alpha_{1}^{'}=\dfrac{ik_{1}}{k}.\end{equation}

 Since $i\mid q-1=ip+s-1$, we obtain $i\mid s-1=(k_{1}-1)(p-s)$.

Further, as $gcd(k_{1}-1,i)=gcd(k_{1}-1,(i+1)k_{1}-1)=k$, it follows that $m=\dfrac{i}{k}\mid \dfrac{k_{1}-1}{k}(p-s)$. Hence $m\mid p-s$ since $gcd(\dfrac{k_{1}-1}{k},\dfrac{i}{k})=1$. So, we deduce by $\eqref{eq130}$ that
\begin{equation}\label{eq14}\alpha_{2}-\alpha_{1}^{'}=k_{1}m\mid k_{1}(p-s)=p-1.\end{equation}
Now, by $\eqref{eq13}$ and $\eqref{eq14}$, we get

$$\left\{\begin{array}{rrr}
\alpha_{2}p-\alpha_{1}^{'}q  & \mid & q-1\\
\alpha_{2}-\alpha_{1}^{'}  & \mid & p-1\\
\end{array}
\right. $$

Therefore, $$\alpha=\dfrac{\alpha_{1}^{'}q }{\alpha_{2}}=\dfrac{(k_{1}-1)q}{(i+1)k_{1}-1}\in \mathbb{Q}\text{-}\mathcal{KS}(N).$$

As in addition $gcd(\alpha_{1}^{'},\alpha_{2})=1$ , $gcd(q,\alpha_{2})=1$ by $\eqref{eq12}$ and $\alpha_{2}\neq1$, we deduce that  $$\dfrac{(k_{1}-1)q}{(i+1)k_{1}-1}\in(\mathbb{Q}\setminus\mathbb{Z})\text{-}\mathcal{KS}(N).$$

\end{proof}

Now,  it remains to prove that each $N$-Korselt base $\beta\in\mathbb{Z}$ generates an  $N$-Korselt base in $\mathbb{Q}\setminus\mathbb{Z}$  where $gcd(\beta,p)=1$, $2p<q<4p$ and $\beta\neq q+p-1$. This is equivalent to  discuss only the cases when $\beta\in \{3q-5p+3,\dfrac{2p+q-1}{2}, q-p+1$\}. It follows, by Corollary~\ref{congene} that we  restrain our work only for $\beta=q-p+1$.

 Assume for the next result, that $2p<q<4p$ and $gcd(q+1,p)=1$.
\begin{proposition} \label{gene4}

Suppose that  $2p<q<4p$. If $q-p+1\in\mathbb{Z}$-$\mathcal{KS}(N)$, then
  $\dfrac{pq}{2p-1}\in(\mathbb{Q}\setminus\mathbb{Z})$-$\mathcal{KS}(N)$.

\end{proposition}

\begin{proof}

     First, if $i=3$, then by Theorem ~\ref{structure1}, we must have $q=4p-3$ and so it's easy to verify that $\dfrac{pq}{2p-1}$ is an  $N$-Korselt base. Further, since $gcd(pq,2p-1)=1$ and $2p-1\neq1$, we get $\dfrac{pq}{2p-1}\notin \mathbb{Z}$. So, we conclude that   $\dfrac{pq}{2p-1}\in(\mathbb{Q}\setminus\mathbb{Z})$-$\mathcal{KS}(N)$.

    Now, suppose that $q=2p+s$. Then $s$ is odd and so $s\neq p-1 $. Assume that $ q-p+1\in \mathbb{Z}$-$\mathcal{KS}(N)$. Then $ s+1\mid p(q-1)$.
      Since in addition, $s<p-1$ hence $gcd(p,s+1)=1$, it follows that  $ s+1 \mid q-1$. Hence, by taking $\alpha_{1}^{''}=1$ and $ \alpha_{2}=2p-1$, we obtain
     $ \alpha_{2}p-\alpha_{1}^{''}pq=-p(s+1)\mid p(q-1)$. Thus, as $ \alpha_{2}q-\alpha_{1}^{''}pq=q(p-1)$, we can write
      $$\left
      \{\begin{array}{rrr}
       \alpha_{2}p-\alpha_{1}^{''}pq  & \mid & p(q-1)\\
       \alpha_{2}q-\alpha_{1}^{''}pq  & \mid & q(p-1)\\
      \end{array}
       \right. $$

     This implies that  $\dfrac{pq}{2p-1}$ is an  $N$-Korselt base.

   Now, as $gcd(pq,2p-1)=gcd(q,q-1-s)=gcd(q,s+1)=1$ and $2p-1\neq1$, we deduce that $\dfrac{pq}{2p-1}\notin \mathbb{Z}$. Thus,  $$\dfrac{pq}{2p-1}\in(\mathbb{Q}\setminus\mathbb{Z})\text{-}\mathcal{KS}(N).$$

\end{proof}

\begin{example}
 Let $N=2*7$ then $\mathbb{Z}$-$\mathcal{KS}(N)=\{6,8\}$ and
 $(\mathbb{Q}\setminus\mathbb{Z})$-$\mathcal{KS}(N)=\left\{\dfrac{7}{2}\right\}$ is exactly the set generated by  $\mathbb{Z}$-$\mathcal{KS}(N)$.

 However, for $N=3*7$, we have  \,  $\mathbb{Z}$-$\mathcal{KS}(N)=\{5,6,9\}$ and
 $(\mathbb{Q}\setminus\mathbb{Z})$-$\mathcal{KS}(N)=\left\{\dfrac{7}{2},\dfrac{7}{3},\dfrac{21}{5},\dfrac{21}{4},\dfrac{15}{2},\dfrac{33}{5}\right\}$, which is composed by more

 than the $N$-Korselt bases in $(\mathbb{Q}\setminus\mathbb{Z})$ generated by  $\mathbb{Z}$-$\mathcal{KS}(N)$.

\end{example}
\begin{proof}(of \textbf{Main Theorem})
By Propositions~\ref{gene1}, ~\ref{gene2}, ~\ref{gene3} and ~\ref{gene4} it follows immediately  that  if $(\mathbb{Q}\setminus\mathbb{Z})$-$\mathcal{KS}(N)=\emptyset$ then $\mathbb{Z}$-$\mathcal{KS}(N)=\{ q+p-1\}$.

\end{proof}
\begin{remark}
The converse of  Main Theorem is not true, for instance if $N=6=2*3$, we have   $\mathbb{Q}$-$\mathcal{KS}(N)=\left\{ 4,\dfrac{3}{2},\dfrac{10}{3},\dfrac{14}{5},\dfrac{8}{3},\dfrac{5}{2},\dfrac{18}{7},\dfrac{12}{5},\dfrac{9}{4}\right\}$.
\end{remark}
This work can motivate us to begin a research in order to  investigate more the rational Korselt set of a number $N$ with more than two prime factors. We believe that the study of an eventual relation(s) between $(\mathbb{Q}\setminus\mathbb{Z})$-$\mathcal{KS}(N)$ and $\mathbb{Z}$-$\mathcal{KS}(N)$  can simplify the task but not enough. The simple case when $N=pq$ is still full of unsolved issues, for instance, after looking to the korselt sets over $\mathbb{Q}$  of some values of $N=pq$ , we  state the following conjecture.

\begin{conjecture}
 For all $N=pq$,  we have $\mathbb{Q}$-$\mathcal{KW}(N)$ is odd.
\end{conjecture}

\end{document}